\documentclass{siamart220329}

\usepackage[hyperref]{bi-discrete}

\newsiamremark{remark}{Remark}

\usepackage[english]{babel}
\defineshorthand[english]{"-}{\babelhyphen{hard}}
\addto\extrasenglish{
  \languageshorthands{english}
  \useshorthands{"}
}

\usepackage{csquotes}

\usepackage{enumerate}
\usepackage[backend=biber,maxbibnames=5,maxalphanames=5,style=alphabetic,bibencoding=utf8,giveninits,url=false,isbn=false]{biblatex}
\AtBeginBibliography{\small}

\newcommand*\diff{\mathop{}\!\mathrm{d}}

\newcommand{\dx}{\ensuremath{\diff x}}
\newcommand{\dy}{\ensuremath{\diff y}}

\newcommand{\tria}{\mathcal{T}}
\newcommand{\Maxgz}{\mathcal{G}_h}
\newcommand{\frakj}{\mathfrak{j}}

\DeclareMathOperator{\dist}{dist}

\newcommand{\TheTitle}{Pointwise gradient estimate of the Ritz projection}
\newcommand{\ShortTitle}{Pointwise gradient estimate of the Ritz projection}
\newcommand{\TheAuthors}{L.~Diening, J.~Rolfes, A.J.~Salgado}

\headers{\ShortTitle}{\TheAuthors}

\title{\TheTitle}

\author{
  Lars Diening\thanks{Department of Mathematics, University of Bielefeld, Postfach 10 01 31, 33501 Bielefeld, Germany.
    (\email{lars.diening@uni-bielefeld.de}, \url{http://www.bi-discrete.com})}
  \and
  Julian Rolfes\thanks{Department of Mathematics, University of Bielefeld, Postfach 10 01 31, 33501 Bielefeld, Germany.
    (\email{julian.rolfes@uni-bielefeld.de}, \url{http://www.bi-discrete.com})}
  \and
  Abner J.~Salgado\thanks{Department of Mathematics, University of Tennessee, Knoxville, TN 37996, USA.
    (\email{asalgad1@utk.edu}, \url{https://math.utk.edu/people/abner-salgado/})}
}

\date{Draft version of \today}

\ifpdf
\hypersetup{
  pdftitle={\TheTitle},
  pdfauthor={\TheAuthors}
}
\fi

\begin{document}

\maketitle 

\begin{abstract}
Let $\Omega \subset \RRn$ be a convex polytope ($n \leq 3$). The Ritz projection is the best approximation, in the $W^{1,2}_0$-norm, to a given function in a finite element space. When such finite element spaces are constructed on the basis of quasiuniform triangulations, we show a pointwise estimate on the Ritz projection. Namely, that the gradient at any point in $\Omega$ is controlled by the Hardy--Littlewood maximal function of the gradient of the original function at the same point. From this estimate, the stability of the Ritz projection on a wide range of spaces that are of interest in the analysis of PDEs immediately follows. Among those are weighted spaces, Orlicz spaces and Lorentz spaces.
\end{abstract}



\section{Introduction}
\label{sec:Intro}

To approximate solutions of partial differential eiquations, in particular, those that are second order and elliptic, the finite element method has emerged as the method of choice. A finite element scheme is nothing but a Galerkin approximation with particular choice of finite dimensional subspace (piecewise polynomials subject to a triangulation of the domain) and a particular basis. It is fair to say that the study the properties of finite element schemes for second order linear elliptic second order equations in an energy setting has reached a state of maturity. In short the Ritz projection, that is the best approximation in the $W^{1,2}_0$-norm (see Section~\ref{sec:Notation} for notation), possesses optimal approximation properties when these are measured in the energy norm, which usually is a norm equivalent to the $W^{1,2}$-norm. This reduces the numerical analysis of a finite element scheme to a question of approximation theory, and this is usually resolved by constructing a suitable interpolant.

On the other hand, the study of the properties of the Ritz projection in non energy norms has been the subject of intensive study with many classical results, recent progresses, and still some open questions. We refer the reader to the Introductions of \cite{GuzLeyRosSch09} and \cite{DemLeySchWah12} for some historical accounts. It is fair to say that the development of this subject is obscured by technicalities, and it is far from settled. Nevertheless, apart from the intrinsic interest such estimates may present, these become important when dealing, for instance, with nonlinear or coupled problems, or even when in a linear problem the data is sufficiently rough that the functional setting that provides well--posedness is no longer the energy one, see for instance \cite{DreDurOje20}, or when the energy norm is not equivalent to the usual {$W^{1,2}$-norm}, see \cite{MR3348172}.

The purpose of this work is to make a contribution in this direction. We show that, over quasiuniform meshes, the gradient of the Ritz projection at any point in the domain is controlled by the Hardy--Littlewood maximal operator of the gradient of the original function at the same point. This pointwise estimate not only immediately implies stability of the Ritz projection in any function space where the maximal operator is bounded, but it also elucidates the action of the Ritz projection, i.e., finite element approximation. It is a sort of averaging procedure. 

Our presentation is organized as follows. In Section~\ref{sec:Notation} we introduce notation. The statement of our main result, Theorem~\ref{theorem:main}, is presented in Section~\ref{sec:Statement}. Here we also colllect a list of Corollaries. Some of these recover known results, whereas others are truly new and may find application in the finite element approximation of, for instance, nonlinear elliptic problems with nonstandard growth conditions \cite{MR2418205}. The proof of our main result is the content of Section~\ref{sec:proof-main-result}. For clarity, this proof is split in several steps that comprise the bulk of this section.

\section{Notation and preliminaries}
\label{sec:Notation}

We begin by introducing some notation and specifying the framework under which we shall operate. The relation $A \lesssim B$ means that there is a constant $c$ for which $A \leq c B$. The value of this constant may change at each occurrence. More importantly, this constant does not depend on $A$, $B$, nor discretization parameters. $A \approx B$ means that $A \lesssim B$ and $B \lesssim A$.

Throughout our work, $\Omega \subset \RRn$, $n \leq 3$, is a bounded convex polytope. While convexity is essential for our arguments, the dimensional restriction is merely an artifact of our methods. Given $x \in \RRn$, we denote its Euclidean norm by $\abs{x}$. By $B(x,r)$ we denote the open ball with center $x \in \RRn$ and radius $r>0$. For a measurable set $E \subset \RRn$ we denote by $|E|$ its Lebesgue measure. $L^0(\Omega)$ denotes the collection of functions $\Omega \to \RR$ that are measurable. For $p \in [1,\infty]$ and $k \in \mathbb{N}$ we denote by $L^p(\Omega)$  and $W^{k,p}(\Omega)$, respectively, the usual Lebesgue and Sobolev spaces. The subspace of $W^{k,p}(\Omega)$ that consists of functions vanishing on the boundary is denoted by $W^{k,p}_0(\Omega)$. We immediately notice that, whenever $w \in W^{k,p}_0(\Omega)$, its extension to $\RRn \setminus \Omega$ by zero, denoted by $\tilde w$, is such that $\tilde w \in W^{k,p}(\RRn)$. For this reason, whenever necessary, we shall make this extension by zero without explicit mention nor change of notation. By $L^1_\loc(\RRn)$ we denote the space of locally integrable functions. For $f \in L^0(\RRn)$ the (centered) Hardy--Littlewood maximal operator $M$ of $f$ is 
\begin{equation}
\label{eq:def_max_op}
  M[f](x) = \sup_{r>0} \frac1{\abs{B(x,r)}} \int_{B(x,r)} \abs{f(y)} \dy.
\end{equation}
With this notation $M[f]$ readily extends to vector valued functions. If $X$ is a normed space, we shall denote by $\| \cdot \|_X$ its norm. If this norm comes from an inner product, this will be denoted by $\skp{\cdot}{\cdot}_X$. We shall make no distinction between scalar and vector valued functions nor their spaces, as this will be clear from context. For $\alpha \in (0,1]$ we let $C^{0,\alpha}(\overline\Omega)$ denote the space of H\"older continuous functions with seminorm
\begin{equation}
\label{eq:def_holder_semi}
  \abs{f}_{C^{0,\alpha}(\overline\Omega)} = \sup_{x,y\in\overline\Omega} \frac{\abs{f(x) - f(y)} }{ \abs{x-y}^\alpha },
\end{equation}
and norm $\| f \|_{C^{0,\alpha}(\overline\Omega)} = \| f \|_{L^\infty(\Omega)} + \abs{f}_{C^{0,\alpha}(\overline\Omega)}$.

Let $\mathbb{T} = \{ \tria_h\}_{h>0}$ be a quasiuniform family of conforming triangulations of $\Omega$ where, for $h>0$, the triangulation $\tria_h$ has mesh size $h$. For $k \in \mathbb{N}$ we denote by
\[
  \mathcal{L}^1_k(\tria_h) = \left\{ w_h \in C(\overline\Omega) : w_{h|T} \in \mathbb{P}_k , \ \forall T \in \tria_h \right\}
\]
the Lagrange space of degree $k$, where $\mathbb{P}_k$ is the space of polynomials of degree at most $k$. We set $V_h = \mathcal{L}_k^1(\tria_h) \cap W^{1,1}_0(\Omega)$ and immediately observe that $V_h \subset W^{1,\infty}_0(\Omega)$. The Ritz projection $R_h : W^{1,1}_0(\Omega) \to V_h$ is defined by
\begin{equation}
\label{eq:def_ritz}
  \skp{ \nabla R_h u}{ \nabla \phi_h }_{L^2(\Omega)} = \skp{ \nabla u}{ \nabla \phi_h }_{L^2(\Omega)}, \qquad \forall \phi_h \in V_h.
\end{equation}
We comment that this mapping is the orthogonal projections onto $V_h$ with respect to the $W^{1,2}_0(\Omega)$-seminorm. The following local error estimate for $R_h$ can be found in \cite[Theorem 1]{DemLeySchWah12}. In fact, it holds for more general families of triangulations than quasiuniform ones.

\begin{proposition}[local error estimate]
\label{proposition:DemlowSchatz}
Let $w \in W^{1,\infty}_0(\Omega)$ and $\mathbb{T}$ be a quasiuniform family of triangulations of $\Omega$. Let $z \in \Omega$ and $h>0$. Define $D = \Omega \cap B_d(z)$ with $d \geq k_0 h$, where $k_0$ is sufficiently large. We have, for every $w_h \in V_h$,
\[
  | \nabla (w - R_h w )(z) |\lesssim \| \nabla (w - w_h) \|_{L^\infty(D)} + d^{-1} \| w - w_h \|_{L^\infty(D)} + d^{-\tfrac{n}2-1} \| w - R_h w \|_{L^2(D)},
\]
where the implicit constant is independent of $w$, $h$, and $z$.
\end{proposition}
\begin{proof}
As mentioned before, this is essentially \cite[Theorem 1]{DemLeySchWah12}. However in that result, as stated, the point $z$ is where $\| \nabla(w - R_h w)\|_{L^\infty(\Omega)}$ is attained. One merely needs to examine the proof to see that this point may be arbitrary.
\end{proof}

\section{Statement of the main result and corollaries}
\label{sec:Statement}

We are now in position to state the main result of our work.

\begin{theorem}[pointwise estimate]
\label{theorem:main}
Let $\Omega \subset \RRn$, for $n \in\{2,3\}$, be a convex polytope and $\mathbb{T} = \{\tria_h\}_{h>0}$ be a family of conforming and quasiuniform triangulations of $\Omega$. For every $u \in W^{1,1}_0(\Omega)$ and almost every $z \in \Omega$ we have
\begin{equation}
\label{eq:mainResult}
  |\nabla R_h u(z)| \lesssim M[\nabla u](z),
\end{equation}
where the implied constant is independent of $z$, $u$, and $h$, and depends on $\mathbb{T}$ only through its shape regularity constants.
\end{theorem}

Before we embark in the proof of this result we immediately mention that it implies the stability of the Ritz projection in any space where the Hardy--Littlewood maximal operator is bounded. For the sake of completeness we present a, far from exhaustive, list of examples: (weighted) $L^p$ spaces, see \S\ref{sub:wLp}, Lorentz spaces, \S\ref{sub:Lorentz} and \S\ref{sub:wLorentz}; Orlicz spaces, \S\ref{sub:Orlicz}; functions of bounded mean oscillation \S\ref{sub:BMO}, and (weighted) variable exponent spaces \S\ref{sub:wLpx}.

\subsection{Lorentz spaces}
\label{sub:Lorentz}
Let $\mu$ be a measure on $\Omega$, $p \in [1,\infty)$, and $q \in [1,\infty]$. The Lorentz spaces are defined as
\[
  L^{p,q}(\mu,\Omega) = \left\{ f \in L^0(\mu,\Omega) : \| f \|_{L^{p,q}(\mu,\Omega)}< \infty \right\},
\]
where
\begin{equation}
\label{eq:LorentzNorm}
  \| f \|_{L^{p,q}(\mu,\Omega)} = \begin{dcases}
    \left( q \int_0^\infty t^{q} \mu_f(t)^{q/p}  \frac{\diff t}{t} \right)^{1/q} , & q < \infty, \\
    \sup_{t>0} t \mu_f(t)^{1/p}, & q = \infty,
  \end{dcases}
\end{equation}
and
\[
  \mu_f(t) = \mu\left(  \{ x \in \Omega: \abs{f(x)} > t \} \right)
\]
is the distribution function of $f$. We recall that, for $p \in [1,\infty)$, $L^{p,p}(\mu, \Omega) = L^p(\mu,\Omega)$ with equivalence of norms \cite[Proposition 1.4.5]{MR3243734}. Finally, if $\mu$ is the Lebesgue measure, we simply denote these spaces by $L^{p,q}(\Omega)$.

\begin{corollary}[Lorentz stability]
\label{cor:LorentzStability}
In the setting of of Theorem~\ref{theorem:main} assume, in addition, that $p \in (1, \infty)$ and $q \in (1,\infty]$, or that $p=1$ and $q=\infty$. Then we have
\[
  \| \nabla R_h u \|_{L^{p,q}(\Omega)} \lesssim \| \nabla u \|_{L^{p,q}(\Omega)},
\]
where the implicit constant is independent of $u$ and $h$. In particular, for $p \in (1,\infty]$, we have
\[
  \| \nabla R_h u \|_{L^{p}(\Omega)} \lesssim \| \nabla u \|_{L^{p}(\Omega)}.
\]
\end{corollary}
\begin{proof}
Consider first the case $p=1$ and $q=\infty$. Owing to, for instance, \cite[Theorem 2.1.6]{MR3243734} we have $M: L^{1,\infty} \to L^{1,\infty}$ boundedly.

For $p>1$ it suffices to invoke \cite[Theorems A, Section 5.2]{Krbec}.
\end{proof}

We comment that the boundedness of the Ritz projection in $L^p$ spaces, has been already presented in \cite{RanSco82,GuzLeyRosSch09,DemLeySchWah12}. Thus, the case $p \in (1,\infty)$ of Corollary~\ref{cor:LorentzStability} can also be obtained by the Marcinkiewicz interpolation theorem as presented in \cite[Theorem 5.3.2]{MR0482275}.

\subsection{Functions of bounded mean oscillation}
\label{sub:BMO}
Let us now present a truly new result. We recall that the (local) sharp maximal operator is defined as
\[
  M^\sharp[f](x) = \sup_{r>0} \int_{B(x,r) \cap \Omega} \abs{ f(y) - \langle f \rangle_{B(x,r) \cap \Omega} } \dy, \qquad \langle f \rangle_{B} = \frac1{|B|} \int_B f(y) \dy.
\]
With its aid we define the space of functions of bounded mean oscillation
\[
  \setBMO(\Omega) = \left\{ f \in L^0(\Omega): M^\sharp[f] \in L^\infty(\Omega) \right\}, \qquad \| f \|_{\setBMO(\Omega)} = \left\| M^\sharp[f] \right\|_{L^\infty(\Omega)}.
\]
We remark that $L^\infty(\Omega) \subsetneq \setBMO(\Omega) \subsetneq L^p(\Omega)$ for all $p \in [1,\infty)$. We refer to \cite[Proposition 3.1.2(2) and Example 3.1.3]{MR3243741} for the first inclusion, and \cite[Corollary 3.1.8]{MR3243741} for the second one.

\begin{corollary}[BMO stability]
\label{cor:BMOStability}
Under the assumptions of Theorem~\ref{theorem:main} we have that
\[
  \| \nabla R_h u \|_{\setBMO(\Omega)} \lesssim \| \nabla u \|_{\setBMO(\Omega)},
\]
where the implicit constant is independent of $u$ and $h$.
\end{corollary}
\begin{proof}
It suffices to recall that, owing to \cite[Theorem 4.2]{MR621018} or \cite[Theorem VII.2.1]{MR2059284}, the Hardy--Littlewood maximal operator is bounded in $\setBMO(\Omega)$.
\end{proof}

\subsection{Orlicz spaces}
\label{sub:Orlicz}
Another new result is stability in Orlicz spaces. We say that $\varphi : (0, \infty)  \to (0,\infty)$ is an Orlicz function if it is nonnegative, increasing, and
\[
  \varphi(0+) = \lim_{t \downarrow 0} \varphi(t)  = 0, \qquad \varphi(\infty) = \lim_{t \to \infty}\varphi(t) = \infty.
\]
If $\varphi$ is an Orlicz function and, in addition, it is convex and satisfies
\[
  \lim_{t \downarrow 0} \frac{\varphi(t)}t = \lim_{t \to \infty} \frac{t}{\varphi(t)} = 0,
\]
then we say that it is an N--function.

For an N--function $\varphi$, we define its corresponding Orlicz space as
\begin{align*}
  L^\varphi(\Omega) &= \left\{ f \in L^0(\Omega): \| f \|_{L^\varphi(\Omega)} < \infty \right\}, \\
  \| f \|_{L^\varphi(\Omega)} &= \inf_{\lambda>0} \left\{ \int_\Omega \varphi\left( \frac1\lambda |f(x)| \right) \dx \leq 1 \right\}.
\end{align*}
We refer the reader to \cite{MR0126722} for further properties of such spaces.

Given an N--function $\varphi$, we say that $\varphi \in \nabla_2$ if there exists $a>1$ such that
\[
  \varphi(t) \leq \frac{1}{2a} \varphi(a t), \qquad \forall t\geq 0.
\]

\begin{corollary}[Orlicz stability]
\label{cor:OrliczStability}
In the setting of Theorem~\ref{theorem:main} let $\varphi \in \nabla_2$. Then
\[
  \| \nabla R_h u \|_{L^\varphi(\Omega)} \lesssim \| \nabla u \|_{L^\varphi(\Omega)},
\]
where the implicit constant is independent of $u$ and $h$.
\end{corollary}
\begin{proof}
According to \cite[{Theorem 1.2.1(v)}]{Krbec}, if $\varphi \in \nabla_2$, then the maximal function is bounded on $L^\varphi(\Omega)$. Apply Theorem~\ref{theorem:main} to conclude.
\end{proof}

\begin{remark}[Simonenko indices]
Given an N--function $\varphi$ define
\[
  h_\varphi(\lambda) = \sup_{t >0} \frac{ \varphi(\lambda t) }{\varphi(t)}, \qquad \lambda >0.
\]
The upper and lower Simonenko indices of $\varphi$ are, respectively,
\[
  p_\varphi^- = \lim_{\lambda \downarrow 0} \frac{ \log h_\varphi(\lambda) }{\log \lambda}, \qquad
  p_\varphi^+ = \lim_{\lambda \to \infty } \frac{ \log h_\varphi(\lambda) }{\log \lambda}.
\]
We comment that $\varphi \in \nabla_2$ implies $p_\varphi^->1$ so that the condition in Corollary~\ref{cor:OrliczStability} is consistent with the results of Corollary~\ref{cor:LorentzStability}.

On the other hand, we say that an N--function is power-like if $p_\varphi^+ < \infty$. According to \cite{MR0199696}, see also \cite{MR1838797}, the space $L^\varphi(\Omega)$ is an intermediate space between $L^p(\Omega)$ and $L^q(\Omega)$ provided the Simonenko indices satisfy
\[
  1 \leq p \leq p_\varphi^- \leq p_\varphi^+ \leq q \leq \infty.
\]
Thus, in the case of $\varphi \in \nabla_2$ and power-like, the results of Corollary~\ref{cor:OrliczStability} could be obtained by interpolation. Since, however, we are not assuming $p_\varphi^+ < \infty$, this is truly a new result.
\end{remark}

\subsection{Muckenhoupt weighted spaces}
\label{sub:wLp}
Next we extend the results of \cite{DreDurOje20} to the optimal range of indices. We recall that a function $0 \leq \omega \in L^1_\loc(\Omega)$ is called a weight. For $p \in [1,\infty)$ we say that a weight $\omega$ belongs to the Muckenhoupt class $\mathcal{A}_p$ if
\[
  [\omega]_{\mathcal{A}_p} = \begin{dcases}
                              \sup_{Q} \left( \frac1{|Q|} \int_Q \omega(x) \dx \right)\left( \frac1{|Q|} \int_Q \omega^{-\tfrac1{p-1}}(x) \dx \right)^{p-1},& p >1, \\
                              \sup_{Q} \left( \frac1{|Q|} \int_Q \omega(x) \dx \right) \| \omega^{-1} \|_{L^\infty(Q)},& p = 1,
                             \end{dcases}
\]
where the supremum is over all cubes $Q$ with sides parallel to the coordinate axes. Weighted Lebesgue spaces are defined, for $p \in (1,\infty)$ and $\omega \in \mathcal{A}_p$, as
\begin{align*}
  L^p(\omega,\Omega) &= \left\{ f \in L^0(\Omega) : \| f \|_{L^p(\omega,\Omega)}< \infty \right\}, \\ 
  \| f \|_{L^p(\omega,\Omega)} &= \left[ \int_\Omega \abs{f(x)}^p \omega(x) \dx \right]^{1/p}.
\end{align*}

\begin{corollary}[weighted stability]
\label{cor:LpwStability}
Under the assumptions of Theorem~\ref{theorem:main} let $p \in (1,\infty)$ and $\omega \in \mathcal{A}_p$. Then, 
\[
  \| \nabla R_h u \|_{L^p(\omega,\Omega)} \lesssim \| \nabla u \|_{L^p(\omega,\Omega)},
\]
where the implicit constant is independent of $u$ and $h$.
\end{corollary}
\begin{proof}
It suffices to recall that, provided $\omega \in \mathcal{A}_p$, the Hardy--Littlewood maximal operator is bounded on weighted spaces; see \cite[Theorem 7.1.9(b)]{MR3243734}.
\end{proof}

As we mentioned above, this result generalizes \cite[Corollary 3.3]{DreDurOje20} where such estimate is obtained, but with $\omega \in \mathcal{A}_{p/2}$, a strictly smaller class.

\subsection{Weighted Lorentz spaces}
\label{sub:wLorentz}

Let $\omega$ be a weight. Here we are concerned with weighted Lorentz spaces $L^{p,q}(\omega,\Omega)$, i.e., the measure in \eqref{eq:LorentzNorm} is $\mu = \omega \dx$.
\begin{corollary}[weighted stability]
\label{cor:wLorentz}
In the setting of Theorem~\ref{theorem:main} let $p \in (1,\infty)$, $q \in (1,\infty]$, and $\omega \in \mathcal{A}_p$. Then
\[
  \| \nabla R_h u \|_{L^{p,q}(\omega,\Omega)} \lesssim \| \nabla u \|_{L^{p,q}(\omega,\Omega)},
\]
where the implicit constant is independent of $u$ and $h$.
\end{corollary}
\begin{proof}
  According to \cite[Theorem 5.2.1]{Krbec}, given the range of exponents, we have that $M :L^{p,q}(\omega,\Omega) \to L^{p,q}(\omega,\Omega)$ boundedly if $\omega \in \mathcal{A}_p$.
\end{proof}

\subsection{Weighted variable exponent spaces}
\label{sub:wLpx}
As a final application we mention weighted variable exponent spaces. A variable exponent is $p \in L^0(\Omega)$ such that $p(\Omega) \subset [1,\infty]$. Given a variable exponent and a weight $0 \leq \omega \in L^1_\loc(\Omega)$ we define weighted variable exponent Lebesgue spaces as
\begin{align*}
  L^{p(\cdot)}_\omega(\Omega) &= \left\{ f \in L^0(\Omega) : \| f \|_{L^{p(\cdot)}_\omega(\Omega)} < \infty\right\}, \\
  \| f \|_{L^{p(\cdot)}_\omega(\Omega)} &= \inf_{\lambda>0} \left\{ \int_\Omega \left| \frac1\lambda {f(x)} \omega(x)\right|^{p(x)} \dx \leq 1 \right\}.
\end{align*}
We refer the reader to \cite{MR3026953,MR2790542} for an extensive treatise on these spaces.

Given a variable exponent $p$, we say that $p \in \mathcal{P}^{\log}(\Omega)$ if 
\[
  \left|\frac1{p(x)} - \frac1{p(y)} \right| \lesssim \frac1{\log( \mathrm{e} + 1/\abs{x-y} )}, \quad \forall x,y \in \Omega,
\]
and, there is $p_\infty \geq 1$ such that
\[
  \left|\frac1{p(x)} - \frac1{p_\infty} \right| \lesssim \frac1{\log( \mathrm{e} + 1/\abs{x} )}, \quad \forall x \in \Omega.
\]
If $p$ is a variable exponent, then $p'$ is its H\"older conjugate; that is, the variable exponent that satisfies
\[
  \frac1{p(x)} + \frac1{p'(x)} = 1,
\]
for almost every $x \in \Omega$. We say that the weight $\omega$ satisfies the generalized Muckenhoupt condition, denoted by $\omega \in \mathcal{A}$, if
\[
  \| \chi_Q  \|_{L^{p(\cdot)}_\omega(\Omega)} \| \chi_Q \|_{L^{p'(\cdot)}_{\omega^{-1}}(\Omega)} \approx |Q|,
\]
for every cube $Q$ with sides parallel to the coordinate axes. Here $\chi_Q$ is the characteristic function of $Q$.

\begin{remark}[$\mathcal{A}$ vs.~$\mathcal{A}_p$]
If $p(x) = p \in (1,\infty)$ for all $x\in \Omega$, then it is known that
\[
  \| f \|_{L^{p(\cdot)}_\omega(\Omega)}^p = \int_\Omega |f(x) \omega(x)|^p \dx = \| f \omega \|_{L^p(\Omega)} = \| f \|_{L^p(\omega^p,\Omega)}.
\]
Thus, we see that $\omega\in \mathcal{A}$ is equivalent to $\mu = \omega^{p} \in \mathcal{A}_p$.
\end{remark}

\begin{corollary}[variable exponent stability]
\label{cor:LpxStability}
Under the assumptions of Theorem~\ref{theorem:main} let $p \in \mathcal{P}^{\log}(\Omega)$ with $\essinf_{x \in \Omega} p(x) > 1$, and $\omega \in \mathcal{A}$. Then
\[
  \| \nabla R_h u \|_{L^{p(\cdot)}_\omega(\Omega)} \lesssim \| \nabla u \|_{L^{p(\cdot)}_\omega(\Omega)},
\]
where the implicit constant is independent of $u$ and $h$.
\end{corollary}
\begin{proof}
Under the given assumptions the Hardy--Littlewood maximal operator is bounded on $L^{p(\cdot)}_\omega(\Omega)$; see \cite{MR2837636}, \cite[Theorem 4.3.8]{MR2790542} and \cite[Theorem 5.8.6]{MR2790542}.
\end{proof}

\subsection{Other extensions and variations}
As we mentioned after Theorem~\ref{theorem:main}, the list we have provided is not exhaustive. For instance, under certain conditions, one can also assert the boundedness in Orlicz--Musielak spaces \cite{MR3811530}.

On the other hand, there are some spaces where the stability remains open. Notable examples are $\mathcal{H}^1(\Omega)$, the atomic Hardy space, and $L^1(\Omega)$.

\section{Proof of the main result}
\label{sec:proof-main-result}

We now focus on the proof of Theorem~\ref{theorem:main}. The technique that we shall follow will be a combination of weighted norm inequalities as in \cite{RanSco82}, and local estimates, as presented in \cite{DemLeySchWah12}. We shall also rely on some estimates on the Green's function that hold, for $n \in \{2,3\}$, in convex polytopes.

\begin{proposition}[Green's function estimates]
\label{proposition:HolderOnGreen}
Let $\Omega \subset \RRn$, with $n \in \{2,3\}$, be a convex polytope and $G : \overline\Omega \times \overline\Omega \to \RR$ be the Green's function associated to this domain. Then, for every $i \in \{1, \ldots, n\}$,
\[
  |\partial_{x_i} G(x,\xi) | \lesssim \frac1{|x-\xi|^{n-1}}, \qquad \forall x,\xi \in \overline\Omega.
\]
In addition, there is $\alpha \in (0,1]$, depending only on the inner angles of $\Omega$, such that, for every $i,j \in \{1, \ldots, n\}$ and all $x,y,\xi \in \Omega$, we have
\begin{align*}
   \frac{\left|\partial_{x_i} G(x,\xi) - \partial_{y_i}G(y,\xi)\right|}{\abs{x-y}^\alpha} &\lesssim \abs{x-\xi}^{-n-\alpha+1} + \abs{y-\xi}^{-n+\alpha+1}, \\
   \frac{\left|\partial_{x_i}\partial_{\xi_j} G(x,\xi) - \partial_{y_i}\partial_{\xi_j}G(y,\xi)\right|}{\abs{x-y}^\alpha} &\lesssim \abs{x-\xi}^{-n-\alpha} + \abs{y-\xi}^{-n+\alpha}.
\end{align*}
\end{proposition}
\begin{proof}
The first bound can be found in \cite[Theorem 3.3(iv)]{GruWid82} for $n\geq 3$, and \cite[Proposition 1 (9)]{MR1156467} for $n=2$.

In the case $n=3$, the H\"older estimates on the first and mixed derivatives can be found in \cite[formula (1.4)]{GuzLeyRosSch09}. When $n=2$ \cite[Lemma 2.1]{DreDurOje20} presents a proof of the estimate for the mixed derivative. The estimate on the first derivative follows the same proof presented in \cite{GuzLeyRosSch09}.
\end{proof}

Notice that the H\"older estimates on derivatives of $G$ is the only moment where our dimensional restriction plays a role. As soon as the estimates in Proposition~\ref{proposition:HolderOnGreen} are valid for more dimensions, the proof of Theorem~\ref{theorem:main} follows \emph{verbatim}.

\subsection{Approximation of identity}
The technique of weighted norms \cite{MR474884,MR568857} relies on the construction of a regularized distance function and its properties. Here we rephrase some of the properties of such function that may help elucidate the reason for its use. For $K,\gamma >0$ to be chosen we define $\varphi_1 : \RRn \to \RR$ by
\begin{equation}
\label{eq:def_phi}
  \varphi_1(x) = c_1 \left( \abs{x}^2 + K^2 \right)^{-\tfrac{n+\gamma}2},
\end{equation}
where $c_1$ is such that $\int_\Omega \varphi_1(x) \dx = 1$. Now, for $\epsilon>0$ and $z \in \Omega$, we define
\begin{align*}
  \varphi_\epsilon(x) &= \epsilon^{-n}\varphi_1(x/\epsilon) = c_1 \epsilon^{-n} \left( \abs{x/\epsilon}^2 + K^2 \right)^{-\tfrac{n+\gamma}2} = c_1 \epsilon^\gamma \left( \abs{x}^2 + K^2\epsilon^2 \right)^{-\tfrac{n+\gamma}2}, \\
  \varphi_{\epsilon,z}(x) &= \varphi_\epsilon(z-x).
\end{align*}
Notice that the family $\{\varphi_\epsilon \}_{\epsilon>0}$ is an approximation of the identity.

\begin{lemma}[convolution estimate]
\label{lemma:convolution}
For every $\epsilon>0$, $z \in \Omega$, and $f \in L^0(\Omega)$ we have
\[
  \| \varphi_{\epsilon,z} f \|_{L^1(\Omega)} = \left( \varphi_\epsilon * |f| \right)(z) \lesssim M[f](z),
\]
where the constant is independent of $\epsilon$, $z$, and $f$.
\end{lemma}
\begin{proof}
Since the function $\varphi_1$ is radial and decreasing, it suffices to invoke Theorem 2.2 of Section 2.2 in \cite{Stein1970}, see also \cite[Theorem 2.1.10]{MR3243734}.
\end{proof}

\subsection{Regularized Green's function}
To establish our main estimate we shall rely on a pointwise representation. We fix $h>0$ and let $z \in \Omega$ be such that $z \in \mathring{T}$ for some $T \in \tria_h$. Owing to shape regularity, there is a function $\delta_z \in C_0^\infty(T)$ such that
\[
  \int_T \delta_z(x) P(x)\dx = P(z), \quad \forall P \in \mathbb{P}_k, \qquad \| D^m \delta_z \|_{L^\infty(\Omega)} \lesssim h^{-n-m}, \quad m \in \setN_0.
\]
Fix $l \in \{1, \ldots, n\}$. The regularized Green's function is $g_z \in W^{1,2}_0(\Omega)$ such that
\begin{equation}
\label{eq:def:reg:green}
  \skp{ \nabla g_z}{\nabla v}_{L^2(\Omega)} = \skp{\delta_z}{\partial_l v }_{L^2(\Omega)}, \quad \forall v \in W^{1,2}_0(\Omega).
\end{equation}

Owing to the fact that the right hand side in \eqref{eq:def:reg:green} is compactly supported in $\Omega$ we can, using Proposition~\ref{proposition:HolderOnGreen}, obtain some H\"older regularity for $g_z$. This is the content of the following result.

\begin{proposition}[estimates on $g_z$]
\label{proposition:Holdergz}
Let $z \in \mathring{T} \in \tria_h$ and $g_z$ solve \eqref{eq:def:reg:green}. Then, for every $i \in \{1,\ldots, n\}$ and all $x,y \notin T$, $x \neq y$, we have
\[
  \frac{ |\partial_{i} g_z(x) - \partial_i g_z(y)|}{|x-y|^\alpha} \lesssim \max_{\xi \in T} \left( |x-\xi|^{-n-\alpha} + |y-\xi|^{-n-\alpha} \right),
\]
where the exponent $\alpha \in (0,1]$ is the same as in Proposition~\ref{proposition:HolderOnGreen}. Moreover,
\[
  \| \nabla g_z \|_{L^\infty(\Omega)} \lesssim h^{-n}.
\]
\end{proposition}
\begin{proof}
We begin by using the pointwise representation of $g_z$ in terms of the Green's function $G$, and the fact that $\delta_z$ is supported on $T$ to obtain
\[
  \partial_i g_z(x) - \partial_i g_z(y) = - \int_T \left( \partial_{x_i} G(x,\xi) - \partial_{y_i} G(y,\xi) \right) \partial_l \delta_z(\xi) \diff \xi.
\]
We now invoke Proposition~\ref{proposition:HolderOnGreen} to obtain
\begin{align*}
  \frac{ |\partial_{i} g_z(x) - \partial_i g_z(y)|}{|x-y|^\alpha} &\leq \sup_{\xi \in T} \frac{ |\partial_{\xi_l}\partial_{x_i} G(x,\xi) - \partial_{\xi_l} \partial_{y_i} G(y,\xi)|}{|x-y|^\alpha} \| \delta_z \|_{L^1(T)} \\
  &\lesssim \max_{\xi \in T} \left( |x-\xi|^{-n-\alpha} + |y-\xi|^{-n-\alpha} \right),
\end{align*}
as claimed.

To obtain the second estimate we observe that $\delta_z$ is supported on $T$ and use its scaling properties to assert that, for any $i \in \{1, \ldots, n \}$, we have
\begin{align*}
  |\partial_i g_z(x) | = \left| \int_T \partial_{x_i} G(x,\xi) \partial_l \delta_z(\xi) \diff \xi \right| \lesssim \int_T |x-\xi|^{1-n} h^{-n-1} \diff \xi \lesssim h^{-n}.
\end{align*}

All estimates have been proved.
\end{proof}

\newcounter{Paso}
\stepcounter{Paso}
\subsection{Step \thePaso: Pointwise representation}

We now begin with the proof of Theorem~\ref{theorem:main} \emph{per se}. Owing to the properties of $\delta_z$ we have that
\begin{align*}
  \partial_l R_h u(z) &= \skp{\delta_z}{\partial_l R_h u}_{L^2(\Omega)} = \skp{\nabla g_z }{\nabla R_h u}_{L^2(\Omega)} =\skp{\nabla R_h g_z }{\nabla u}_{L^2(\Omega)} \\
  &=  \skp{\delta_z}{\partial_l u}_{L^2(\Omega)}+ \skp{\nabla(R_h g_z - g_z)}{\nabla u}_{L^2(\Omega)}
\end{align*}
where we used \eqref{eq:def:reg:green} and the definition of the Ritz projection \eqref{eq:def_ritz}. From the definition of $\delta_z$ it follows immediately that
\[
  \left|\skp{\delta_z}{\partial_l u}_{L^2(\Omega)} \right| \lesssim M[\nabla u](z).
\]
On the other hand, we estimate the second term as
\[
  \left| \skp{\nabla(R_h g_z - g_z)}{\nabla u}_{L^2(\Omega)}\right| \leq \left\| \varphi_{h,z} \nabla u \right\|_{L^1(\Omega)} \left\| \varphi_{h,z}^{-1} \nabla(R_h g_z - g_z) \right\|_{L^\infty(\Omega)}.
\]
Owing to Lemma~\ref{lemma:convolution},
\[
  \left\| \varphi_{h,z} \nabla u \right\|_{L^1(\Omega)} \lesssim M[\nabla u](z).
\]
Thus, if we define
\begin{equation}
\label{eq:def_Mh}
  \Maxgz = \sup_{z \in \Omega} \left\| \varphi_{h,z}^{-1} \nabla(R_h g_z - g_z) \right\|_{L^\infty(\Omega)},
\end{equation}
we see that the heart of the matter is to provide a uniform, in $h$, estimate for this quantity.

In summary, the rest of the proof consists in showing the following result.

\begin{proposition}[uniform estimate]
\label{proposition:UniformMh}
In the setting of Theorem~\ref{theorem:main} there are $K>k_0$ and $\gamma \in (0,\alpha)$ such that, if $\varphi_1$ is defined as in \eqref{eq:def_phi}, we have
\[
  \Maxgz \lesssim 1,
\]
where the constant is independent of $h>0$, and $\Maxgz$ was defined in \eqref{eq:def_Mh}. Here $k_0$ is as in Proposition~\ref{proposition:DemlowSchatz}, and $\alpha$ as in Proposition~\ref{proposition:HolderOnGreen}.
\end{proposition}

\stepcounter{Paso}
\subsection{Step \thePaso: Dyadic decomposition}

Fix $z \in \Omega$. Define, for $j \in \setN_0$, $d_j = 2^j K h$. We decompose the domain $\Omega$ into the following annuli
\begin{equation}
\label{eq:annuli}
  \begin{aligned}
    B_j &= \left\{ x \in \Omega: |x-z| < d_j \right\}, & A_j &= B_j \setminus B_{j-1}, \\
    A_j^+ &= B_{j+1} \setminus B_{j-2}, & A_j^{++} &= B_{j+2} \setminus B_{j-3},
  \end{aligned}
\end{equation}
with the convention that, for $j<0$, $B_j = \emptyset$. For $S \subset \overline\Omega$ we also define
\[
  \mathcal{N}_h(S) = \bigcup \left\{ T\in \tria_h : S \cap T \neq \emptyset \right\}.
\]
The use of this dyadic decomposition lies on the fact that on each annulus the regularized distance function $\varphi_{h,z}$ is almost constant.

\begin{lemma}[distance estimates]
\label{lemma:distance}
Assume that $K>2$. For all $j\geq0$ we have
\[
  \phi_{h,z}(x) \approx h^\gamma d_j^{-n-\gamma}, \qquad \forall x\in A_j, 
\]
and
\[
  \dist\left(A_j^+, \mathcal{N}_h(\Omega\setminus A_j^{++}) \right) \approx d_j,
\]
where the implicit constants are independent of $h$. As a consequence, for $j \geq 3$ we have,
\[
  \left| \nabla g_z \right|_{C^{0,\alpha}(A_j^{++})} \lesssim d_j^{-n-\alpha},
\]
where $\alpha \in (0,1]$ is as in Proposition~\ref{proposition:HolderOnGreen}.
\end{lemma}
\begin{proof}
The estimate on $\phi_{h,z}$ follows by definition. The estimate on the distance between $A_j^+$ and $\mathcal{N}_h(\Omega\setminus A_j^{++})$ does so as well.

On the other hand, if $j \geq 3$, $x,y \in A_j^{++}$, and $\xi \in T$ then $\abs{x-\xi}, \abs{y-\xi} \approx d_j$. We can then refine the estimate of Proposition~\ref{proposition:Holdergz} to conclude 
\[
  \frac{ |\partial_{i} g_z(x) - \partial_i g_z(y)|}{|x-y|^\alpha} \lesssim d_j^{-n-\alpha}. 
\]
\end{proof}

\stepcounter{Paso}
\subsection{Step \thePaso: Reduction to interpolation and duality}

Let $\frakj \in \setN_0$ now be such that
\[
  \Maxgz = \left\| \varphi_{h,z}^{-1} \nabla(g_z - R_h g_z) \right\|_{L^\infty(A_\frakj)}.
\]
Using the distance estimates of Lemma~\ref{lemma:distance} we can also assert that
\[
  \Maxgz \lesssim h^{-\gamma} d_\frakj^{n+\gamma} \left\| \nabla(g_z - R_h g_z) \right\|_{L^\infty(A_\frakj)}.
\]
Now choose $K \geq k_0$, where $k_0$ was introduced in Proposition~\ref{proposition:DemlowSchatz}. Then we have $\diameter A_\frakj \approx d_j \geq k_0 h$, so that with a simple covering argument we may obtain that
\begin{equation}
\label{eq:MhleqIpIIpIII}
\begin{aligned}
  \Maxgz &\lesssim h^{-\gamma} d_\frakj^{n+\gamma} \left\| \nabla(g_z - R_h g_z) \right\|_{L^\infty(A_\frakj)}  \\
  &\lesssim h^{-\gamma} d_\frakj^{n+\gamma} \left( \| \nabla( g_z - \Pi_h g_z ) \|_{L^\infty(A_\frakj^+)} + d_\frakj^{-1} \| g_z - \Pi_h g_z \|_{L^\infty(A_\frakj^+)} \right. \\ &+ \left. d_\frakj^{-\tfrac{n}2-1} \| g_z - R_h g_z \|_{L^2(A_\frakj^+)} \right) 
  = \mathrm{I} + \mathrm{II} + \mathrm{III},
\end{aligned}
\end{equation}
where $\Pi_h$ is, for instance, the so-called Scott--Zhang interpolant \cite{MR1011446}. The first two terms will be handled using interpolation estimates, whereas the last one is controlled by duality.

\stepcounter{Paso}
\subsection{Step \thePaso: Bound of \texorpdfstring{$\mathrm{I} + \mathrm{II}$}{I + II} via interpolation estimates}
It is our goal now to bound $\mathrm{I} + \mathrm{II}$ using the approximation properties of $\Pi_h$ and the regularity of $g_z$. This regularity, however, depends on the distance between $A_\frakj^+$ and $z$. If $\frakj\geq 3$, then we can invoke the estimate in Lemma~\ref{lemma:distance} to see that
\begin{align*}
  \| g_z - \Pi_h g_z \|_{L^\infty(A_\frakj^+)} + h \| \nabla( g_z - \Pi_h g_z ) \|_{L^\infty(A_\frakj^+)} &\lesssim h^{1+\alpha} |\nabla g_z |_{C^{0,\alpha}(A_\frakj^{++})} \\ &\lesssim h^{1+\alpha} d_\frakj^{-n-\alpha}.
\end{align*}
As a consequence, since $0 < \gamma < \alpha$,
\begin{equation}
\label{eq:BoundIpIIvers1}
\begin{aligned}
  \mathrm{I} + \mathrm{II} &\lesssim h^{-\gamma} d_\frakj^{n+\gamma} \left( h^\alpha d_\frakj^{-n-\alpha} + h^{1+\alpha}d_\frakj^{-n-\alpha-1} \right) 
  \lesssim \left( \frac{h}{d_\frakj} \right)^{\alpha-\gamma} + \left( \frac{h}{d_\frakj} \right)^{1+\alpha-\gamma}\\
  &\leq \frac1{K^{\alpha-\gamma}} + \frac1{K^{1+\alpha-\gamma}}.
\end{aligned}
\end{equation}

If, on the other hand, $\frakj < 3$ we use the second bound of Proposition~\ref{proposition:Holdergz} to obtain
\[
  \| g_z - \Pi_h g_z \|_{L^\infty(A_\frakj^+)} + h\| \nabla( g_z - \Pi_h g_z ) \|_{L^\infty(A_\frakj^+)} \lesssim h \|\nabla g_z \|_{L^\infty(\Omega)} \lesssim h^{1-n}.
\]
In this case then we get 
\begin{equation}
 \label{eq:BoundIpIIvers2}
\begin{aligned}
  \mathrm{I} + \mathrm{II} &\lesssim h^{-\gamma} d_\frakj^{n+\gamma} \left( h^{-n} + d_\frakj^{-1} h^{1-n} \right) \lesssim h^{-n-\gamma}d_3^{n+\gamma} + h^{1-n-\gamma}d_3^{n+\gamma-1} \\
  &\lesssim K^{n+\gamma} + K^{n+\gamma-1}.
\end{aligned}
\end{equation}

Gathering \eqref{eq:BoundIpIIvers1} and \eqref{eq:BoundIpIIvers2} we arrive at
\begin{equation}
\label{eq:BoundIpII}
  \mathrm{I} + \mathrm{II} \lesssim \max\left\{ \frac1{K^{\alpha-\gamma}} + \frac1{K^{1+\alpha-\gamma}}, K^{n+\gamma} + K^{n+\gamma-1} \right\}.
\end{equation}

\stepcounter{Paso}
\subsection{Step \thePaso: Bound of \texorpdfstring{$\mathrm{III}$}{III} by duality}
We bound $\mathrm{III}$ by duality. Define
\begin{equation}
\label{eq:DefSjForDuality}
  \mathcal{S}_\frakj = \left\{v \in C_0^\infty(\Omega) : \| v \|_{L^2(\Omega)} \leq 1, \ \support(v) \subset A_j^+ \right\}
\end{equation}
so that
\[
  \| g_z - R_h g_z \|_{L^2(A_\frakj^+)} = \sup_{0 \neq v \in \mathcal{S}_\frakj} \skp{g_z - R_h g_z}{v}_{L^2(\Omega)}. 
\]

Fix $v \in \mathcal{S}_\frakj$ and let $w_v \in W^{1,2}_0(\Omega)$ solve
\begin{equation}
\label{eq:DefOfwvForDuality}
  -\Delta w_v = v, \text{ in } \Omega, \qquad w_v = 0, \text{ on } \partial\Omega.
\end{equation}
Then, by Galerkin orthogonality,
\begin{align*}
  \skp{g_z - R_h g_z}{v}_{L^2(\Omega)} &= \skp{\nabla(g_z - R_h g_z)}{\nabla w_v}_{L^2(\Omega)} \\
    &= \skp{\nabla(g_z - R_h g_z)}{\nabla (w_v -\Pi_h w_v)}_{L^2(\Omega)} \\
    &= \skp{\varphi_{h,z}^{-1} \nabla(g_z - R_h g_z)}{\varphi_{h,z}\nabla (w_v -\Pi_h w_v)}_{L^2(\Omega)}.
\end{align*}
An application of H\"older's inequality then allows us to conclude that
\[
  \mathrm{III} \leq h^{-\gamma}d_\frakj^{\tfrac{n}2+\gamma-1} \sup_{v \in \mathcal{S}_j} \| \varphi_{h,z}\nabla (w_v -\Pi_h w_v)\|_{L^1(\Omega)} \Maxgz .
\]
Notice that if, in this last estimate, the term that is multiplying $\Maxgz$ is sufficiently small, then it could be absorbed on the left hand side in \eqref{eq:MhleqIpIIpIII}. This possibility is explored in the following result.

\begin{lemma}[duality bound]
\label{lemma:DualityBound}
Let $\mathcal{S}_\frakj$ be defined in \eqref{eq:DefSjForDuality} and $\gamma \in (0,\alpha)$. There is a constant, independent of $\frakj$, $z$, and $h$ such that
\[
  h^{-\gamma}d_\frakj^{\tfrac{n}2+\gamma-1} \sup_{v \in \mathcal{S}_j} \| \varphi_{h,z}\nabla (w_v -\Pi_h w_v)\|_{L^1(\Omega)} \leq C \left( \frac1K + \frac1{K^{\alpha-\gamma}} \right),
\]
where $w_v \in W^{1,2}_0(\Omega)$ is the solution to \eqref{eq:DefOfwvForDuality} and $\alpha$ is as in Proposition~\ref{proposition:HolderOnGreen}.
\end{lemma}
\begin{proof}
Let $v \in \mathcal{S}_j$ be arbitrary. Using Lemma~\ref{lemma:distance} and scaling we have
\begin{align*}
  h^{-\gamma}d_\frakj^{\tfrac{n}2+\gamma-1} \| \varphi_{h,z}\nabla (w_v -\Pi_h w_v)\|_{L^1(A_\frakj^{++})} &\lesssim d_\frakj^{-\tfrac{n}2-1} \| \nabla (w_v -\Pi_h w_v)\|_{L^1(A_\frakj^{++})} \\ 
  &\leq d_\frakj^{-1} \| \nabla (w_v -\Pi_h w_v)\|_{L^2(A_\frakj^{++})} \\
  &\lesssim \frac{h}{d_\frakj} |w_v|_{W^{2,2}(\Omega)} \lesssim \frac1K \| v \|_{L^2(\Omega)} \leq \frac1K,
\end{align*}
where, since $\Omega$ is convex, we used a regularity estimate on $w_v$.

To control the norm in $\Omega \setminus A_\frakj^{++}$ observe that, owing to the estimates of Proposition~\ref{proposition:HolderOnGreen}, for every $i \in \{1, \ldots, n\}$ we have
\begin{align*}
  \frac{|\partial_i w_v(x) - \partial_iw_v(y)|}{\abs{x-y}^\alpha} &\leq \int_{A_\frakj^+} \frac{|\partial_{x_i} G(x,\xi) - \partial_{y_i} G(y,\xi)|}{\abs{x-y}^\alpha} |v(\xi)| \diff \xi \\
  &\lesssim \max_{\xi \in A_\frakj^+} \left( |x-\xi|^{-n-\alpha+1} + |y-\xi|^{-n-\alpha+1} \right) \int_{A_\frakj^+} |v(\xi)| \diff \xi \\
  &\lesssim d_\frakj^{-n-\alpha+1} d_\frakj^{\tfrac{n}2} \| v \|_{L^2(A_\frakj^+)} \leq d_\frakj^{1-\alpha-\tfrac{n}2},
\end{align*}
where we used the second distance estimate of Lemma~\ref{lemma:distance}. This shows that
\[
  |\nabla w_v |_{C^{0,\alpha}(\mathcal{N}_h(\Omega\setminus A_\frakj^{++}))} \lesssim d_\frakj^{1-\alpha-\tfrac{n}2}.
\]

To shorten notation let $e_w = w_v -\Pi_h w_v$. We use that $\| \varphi_{h,z} \|_{L^1(\Omega)} = 1$ and the recently obtained regularity estimate to proceed as follows:
\begin{align*}
  h^{-\gamma}d_\frakj^{\tfrac{n}2+\gamma-1} \| \varphi_{h,z}\nabla e_w \|_{L^1(\Omega\setminus A_\frakj^{++})} &\leq h^{-\gamma}d_\frakj^{\tfrac{n}2+\gamma-1} \| \nabla e_w\|_{L^\infty(\Omega\setminus A_\frakj^{++})} \\
  &\leq h^{-\gamma}d_\frakj^{\tfrac{n}2+\gamma-1} h^\alpha  d_\frakj^{1-\alpha-\tfrac{n}2} = \left( \frac{h}{d_\frakj} \right)^{\alpha-\gamma} \leq \frac1{K^{\alpha-\gamma}}.
\end{align*}

We combine both bounds to conclude.
\end{proof}

With Lemma~\ref{lemma:DualityBound} at hand we conclude that
\begin{equation}
\label{eq:BoundIII}
  \mathrm{III} \lesssim \left( \frac1K + \frac1{K^{\alpha-\gamma}} \right) \Maxgz.
\end{equation}

\stepcounter{Paso}
\subsection{Step \thePaso: Final step. Gathering all the estimates}
With the aid of \eqref{eq:BoundIpII} and \eqref{eq:BoundIII}, estimate \eqref{eq:MhleqIpIIpIII} reduces to
\[
  \Maxgz \lesssim \max\left\{ \frac1{K^{\alpha-\gamma}} + \frac1{K^{1+\alpha-\gamma}}, K^{n+\gamma} + K^{n+\gamma-1} \right\} + \left(\frac1K + \frac1{K^{\alpha-\gamma}} \right) \Maxgz,
\]
provided $K \geq k_0$, where $k_0$ is defined in Proposition~\ref{proposition:DemlowSchatz}; and $\gamma \in (0, \alpha)$, with $\alpha$ as in Proposition~\ref{proposition:HolderOnGreen}. We can now, if necessary, choose an even bigger $K$ to conclude the proof of Proposition~\ref{proposition:UniformMh} and, as a consequence, that of Theorem~\ref{theorem:main}.
\section*{Acknowledgements}
The work of AJS is partially supported by NSF grant DMS-2111228.
The main results of this paper were obtained in February 2022 when AJS was visiting Bielefeld University for the conference: ``\emph{Nonlocal Equations: Analysis and Numerics}''. This work was also funded by the Deutsche Forschungsgemeinschaft (DFG, German Research Foundation) - SFB 1283/2 2021 - 317210226.

 \printbibliography

\end{document}